\documentclass[a4paper,12pt]{article}
\usepackage{amssymb,amsmath,amsthm,latexsym}
\usepackage{amsfonts}
\usepackage{amsfonts}
\usepackage{graphicx}
\usepackage[pdftex,bookmarks,colorlinks=false]{hyperref}
\usepackage{verbatim}
\usepackage[font=small,labelfont=bf]{caption}
\usepackage{subcaption}

\newtheorem{theorem}{Theorem}[section]

\newtheorem{definition}[theorem]{Definition}

\newtheorem{proposition}[theorem]{Proposition}

\title{\bf Operations on Covering Numbers of Certain Graph Classes}
\author{{\bf Susanth C\footnote{Department of Mathematics, Research \& Development Centre, Bharathiar University, Coimbatore - 641046, Tamilnadu, email: {\em susanth\_c@yahoo.com}}}   ~and
{\bf Sunny Joseph Kalayathankal\footnote{Department of Mathematics, Kuriakose Elias College, Mannanam, Kottayam - 686561, Kerala, email:{\em sunnyjoseph2014@yahoo.com}}}}
\date{}

\begin{document}
\maketitle
\begin{abstract}
The bounds on the sum and product of chromatic numbers of a graph and its complement are known as Nordhaus-Gaddum inequalities. 
In this paper, we study the operations on the Independence numbers of graphs with their complement.  We also provide a new characterization of certain graph classes.
\end{abstract}
\section{Introduction}
For all  terms and definitions, not defined specifically in this paper, we refer to \cite{FH}. Unless mentioned otherwise, all graphs considered here are simple, finite and have no isolated vertices.
\\Many problems in extremal graph theory seek the extreme values of graph parameters
on families of graphs. The classic paper of Nordhaus and Gaddum \cite{KCA} study the extreme values of the sum (or product) of a parameter on a graph and its complement, following  solving these problems for the chromatic number on n-vertex graphs. In this paper, we study such problems for some graphs and their associated graphs.
\begin{definition}\rm{
\cite{CH} A \textit{Walk}, $W=v_0e_1v_1e_2v_2 \ldots v_{k-1} e_kv_k$, in a graph $G$ is a finite sequence whose terms are alternately vertices and edges such that, for $1\leq i\leq k$, the edge $e_i$ has ends $v_{i-1}$and $v_i$.}
\end{definition}
\begin{definition}\rm{
\cite{CH} If the vertices $v_0,v_1,\ldots,v_k$ of a walk $W$ are distinct then $W$ is called a \textit{Path}.  A path with $n$ vertices will be denoted by $P_n$. $P_n$ has length $n-1$. }
\end{definition}

\begin{definition}\rm{
\cite{CH} Let $G$ be a simple graph with $n$ vertices.  The \textit{complement} $\bar{G}$ of $G$ is defined 
to be the simple graph with the same vertex set as $G$ and where two vertices $u$ 
and $v$ are adjacent precisely when they are not adjacent in $G$.  Roughly speaking 
then, the complement of $G$ can be obtained from the complete graph $K_n$ by 
rubbing out all the edges of $G$. }  
\end{definition}

\begin{definition}\rm{
\cite{CGT} Two vertices that are not adjacent in a graph $G$
are said to be \textit{independent}. A set $S$ of vertices is independent if any two vertices of $S$ are independent. The \textit{vertex independence number} or simply the \textit{independence number}, of a graph $G$, denoted by  $\alpha(G)$ is the maximum cardinality among
the independent sets of vertices of $G$.}
\end{definition} 
\begin{definition}{\rm
\cite{BM1} A subset $M$ of the  edge set of $G$, is called a \textit{matching} in $G$ if no two of the edges in $M$ are adjacent. In other words, if for any two edges $e$ and $f$ in $M$, both the end vertices of $e$ are different from the end vertices of $f$.}
\end{definition}

\begin{definition}{\rm
\cite{BM1} A \textit{perfect matching} of a graph $G$ is a matching of $G$ containing $n/2$ edges, the largest possible, meaning perfect matchings are only possible on graphs with an even number of vertices.  A perfect matching sometimes called a \textit{complete matching} or \textit{1-factor}}.
\end{definition}

\begin{definition}{\rm
\cite{BM1} The matching number of a graph $G$, denoted by $\nu(G)$, is the size of a maximal independent edge set. It is also known as \textit{edge independence number}. The matching number $\nu(G)$  satisfies the inequality $\nu(G)\leq\lfloor \frac{n}{2}\rfloor $. \newline Equality occurs only for a perfect matching and graph $G$ has a perfect matching if and only if $|G|=2~\nu(G)$, where $|G|=n$ is the vertex count of $G$.}
\end{definition}

\begin{definition}{\rm
\cite{BM1}A \textit{maximum independent set} in a line graph corresponds to maximum matching in the original graph. } 
\end{definition}
In this paper, we discussed the sum and product of the independence numbers of certain class of graphs and their line graphs.

\section{New Results}
\begin{theorem}\label{th1} \rm{
\cite{DBW} The independence number $\alpha(G)$ of a graph  $G$ and vertex cover number $\beta(G)$ are related by  $\alpha(G)+\beta(G)=|G|$, 
where $|G|=n$, the vertex count of $G$.}
\end{theorem}

\begin{theorem}\label{th2}{\rm
\cite{DBW} The independence number of the line graph of a graph $G$ is equal to the matching number of $G$.}
\end{theorem}
\begin{proposition}{\rm 
For a complete graph $K_{n}$, $n\geq 3$,
$\beta(K_{n})+\beta(\bar K_n)= n-1$ and $\beta(K_{n}).\beta(\bar K_n)= 0$.

}
\end{proposition}
\begin{proof}
The independence number of a complete graph $K_n$ on $n$ vertices is $1$, since each vertex is joined with every other vertex of the graph $G$. 
By theorem \ref{th1}, the covering number, $\beta(K_n)$ of $K_n$ = $n-1$.  By \cite{SS2}, Since, the complement of $K_n$ is an empty graph with $n$ vertices, it will have an independence number $n$.  Therefore by theorem \ref{th1}, $\beta(\bar{K_n})=n-n=0$.
That is for a complete graph $K_{n}$, $n\geq 3$,
$\beta(K_{n})+\beta(\bar K_n)= n-1$ and $\beta(K_{n}).\beta(\bar K_n)= 0$.
\end{proof}

\begin{proposition}
For a complete bipartite graph $K_{m,n}$, 
\begin{center}
$\beta(K_{m,n})+\beta(\bar K_{m,n})=2(m-1)+n$  and
\end{center}
 \begin{center}
$\beta(K_{m,n}).\beta(\bar K_{m,n})= m(m+n-2)$
 \end{center}
\end{proposition}

\begin{proof}
Without the loss of generality, let $m<n$.  The independence number of a complete bipartite graph, $\alpha(K_{m,n}) = max(m,n)=n$.  Since a complete bipartite graph consists of $m+n$ number of vertices, by theorem \ref{th1}, $\beta(K_{m,n})=m+n-n = m$.  The complement of $K_{m,n}$ is $K_{m}\cup K_{n}$, which is a disjoint union of 2 complete graphs.  The independence number of $K_{m}\cup K_{n}$ is 1+1=2.  By \ref{th1}, $\beta(\bar K_{m,n})= m+n-2$.

Therefore, $\beta(K_{m,n})+\beta(\bar K_{m,n})= m+(m+n-2)=2(m-1)+n$  and $\beta(K_{m,n}).\beta(\bar K_{m,n})= m(m+n-2)$
\end{proof}

\begin{definition}{\rm
\cite{FH} For $n\geq 3$, a \textit{wheel graph} $W_{n+1}$ is the graph $K_1+ C_{n}$.  A wheel graph $W_{n+1}$ has $n+1$ vertices and $2n$ edges.}
\end{definition}

\begin{theorem}{\rm
For a wheel graph $W_{n+1}$, $n\geq 3$, \[\beta(W_{n+1})+\beta(\bar{W_{n+1}})=
\left\{
\begin{array}{ll}
2n+1 & \text{; if $n$ is even} \\
2(n+1) & \text{; if $n$ is odd}
\end{array}\right.\] and
\[\beta(W_{n+1}).\beta(W_{n+1})=
\left\{
\begin{array}{ll}
\frac{3n(n+2)}{4} & \text{; if $n$ is even} \\\\
\frac{(n+3)(3n+1)}{4} & \text{; if $n$ is odd}
\end{array}\right.\]}
\end{theorem}

\begin{proof}
By \cite{SS1}, the independence number of a wheel graph $W_{n+1}$    is $\lfloor \frac{n}{2}\rfloor$.  The number of vertices in $W_{n+1}$ is $n+1$.  Then by theorem \ref{th1}, $\beta(W_{n+1})=(n+1)-\lfloor\frac{n}{2}\rfloor = (n+1)-\frac{n}{2}=\frac{(n+2)}{2}$, if $n$ is even and $(n+1)-\frac{(n-1)}{2}=\frac{(n+3)}{2}$, if $n$ is odd.  Since the central vertex is adjacent to all other vertices of $W_{n+1}$, in $\bar W_{n+1}$, that vertex contributes 1 to the maximal independence set.  All the vertices other than the two which are adjacent in $W_{n+1}$ will be adjacent to any vertex $v$ in $\bar W_{n+1}$.  There fore any maximal independence set in $\bar W_{n+1}$ can have atmost 2 elements.  That is independence number of $\bar W_{n+1}$ is 3. There fore by \ref{th1}, $\beta(\bar W_{n+1})=n+1-3=n-2$.

That is For $n\geq 3$, \[\beta(W_{n+1})+\beta (\bar W_{n+1})=
\left\{
\begin{array}{ll}
2n+1 & \text{; if $n$ is even} \\
2(n+1) & \text{; if $n$ is odd}
\end{array}\right.\] and
\[\beta(W_{n+1}).\beta (\bar W_{n+1}))=
\left\{
\begin{array}{ll}
\frac{3n(n+2)}{4} & \text{; if $n$ is even} \\\\
\frac{(n+3)(3n+1)}{4} & \text{; if $n$ is odd}
\end{array}\right.\]
\end{proof}

\begin{definition}{\rm
\cite{JAG} \textit{Helm graphs} are graphs obtained from a wheel by attaching one pendant edge to each vertex of the cycle. }
\end{definition}

\begin{theorem}
For a helm graph $H_n$, $n\geq 3$,
$\beta(H_n)+\beta \bar{(H_n)}= 3n-2$ and
$\beta(H_n).\beta \bar{(H_n)}= 2n(n-1)$.
\end{theorem}

\begin{proof}
A helm graph $H_n$ consists of $2n+1$ vertices and $3n$ edges. By \cite{SS1}, the independence number of a helm graph, $\alpha(H_n)=n+1$.  Since, the number of vertices of  $H_n$ is $2n+1$, by theorem \ref{th1}, $\beta(H_n)= (2n+1)-(n+1)=n$. By \cite{SS2},$\alpha(\bar{H_n})$ = 3 and by theorem \ref{th1}, $\beta(\bar{H_n})=(2n+1)-3=2n-2$.
\\
Therefore, For a helm graph $H_n$, $n\geq 3$,
$\beta(H_n)+\beta \bar{(H_n)}= 3n-2$ and
$\beta(H_n).\beta \bar{(H_n)}= 2n(n-1)$.
\end{proof}

\begin{definition}{\rm
\cite{DBW} Given a vertex $x$ and a set $U$ of vertices, an $x$, $U-$fan is a set of paths from $x$ to $U$ such that any two of them share only the vertex $x$.  A $U-$fan is denoted by $F_{1,n}$.}
\end{definition}

\begin{theorem}
For a fan graph $F_{1,n}$, 
\[\beta(F_{1,n})+\beta \bar{(F_{1,n}})=
\left\{
\begin{array}{ll}
\frac{3n-2}{2} & \text{; if $n$ is even} \\
\frac{3n-3}{2}  & \text{; if $n$ is odd}
\end{array}\right.\] and
\[\beta(F_{1,n}).\beta\bar{(F_{1,n}})=
\left\{
\begin{array}{ll}
\frac{n^2-4}{2} & \text{; if $n$ is even} \\
\frac{(n+1)(n-2)}{2} & \text{; if $n$ is odd}
\end{array}\right.\]
\end{theorem}
\begin{proof}
A fan graph $F_{1,n}$ is defined to be a graph $K_1+ P_{n}$.  By \cite{SS1}, the independence number of a fan graph $F_{1,n}$  is either $\frac{n}{2}$ or $\frac{n+1}{2}$, depending on $n$ is even or odd.  Since the number of vertices of $F_{1,n}$ is $n+1$, by theorem \ref{th1}, $\beta(F_{1,n})=(n+1)-\frac{n}{2}=\frac{n+2}{2}$, if $n$ is even and $(n+1)-\frac{(n+1)}{2} = \frac{n+1}{2}$, if $n$ is odd.  By \cite{SS2}, the independence number of $\bar F_{1,n}$ is 3.  By theorem \ref{th1}, $\beta (\bar F_{1,n})= n+1-3=n-2$.
\\
Therefore, 
For a fan graph $F_{1,n}$, 
\[\beta(F_{1,n})+\beta \bar{(F_{1,n}})=
\left\{
\begin{array}{ll}
\frac{3n-2}{2} & \text{; if $n$ is even} \\
\frac{3n-3}{2}  & \text{; if $n$ is odd}
\end{array}\right.\] and
\[\beta(F_{1,n}).\beta\bar{(F_{1,n}})=
\left\{
\begin{array}{ll}
\frac{n^2-4}{2} & \text{; if $n$ is even} \\
\frac{(n+1)(n-2)}{2} & \text{; if $n$ is odd}
\end{array}\right.\]
\end{proof}

\begin{definition}{\rm
\cite{AVJ, IS} An $n-$sun or a \textit{trampoline}, denoted by $S_n$, is a chordal graph on $2n$ vertices, where $n \geq 3$, whose vertex set can be partitioned into two sets $U = \{u_1, u_2, u_3, . . . , u_n\}$ and $W = \{w_1, w_2, w_3, . . . , w_n\}$ such that $U$ is an independent set of $G$ and $u_i$ is adjacent to $w_j$ if and only if $j = i$ or $j = i + 1 (\mod{n})$. A \textit{complete sun} is a sun $G$ where the induced subgraph  $\langle U\rangle$ is complete.}
\end{definition}
\begin{theorem}{\rm
For a complete sun graph $S_n$, $n\geq 3$,
$\beta(S_n)+\beta(\bar{S_n})= 2n$ and
$\beta(S_n).\beta(\bar{S_n})= n^2$.}
\end{theorem}
\begin{proof}
Let $S_n$ be a complete sun graph on $2n$ vertices. 
By \cite{SS1}, the independence number of $S_n$, $\alpha(S_n)=n$.  Since $S_n$ consists of $2n$ vertices, by theorem \ref{th1}, $\beta(S_n)=2n-n=n$.
Now consider the complement of $S_n$.  Since $V = \{v_1, v_2, v_3, . . . , v_n\}$ be the vertex set of the complete graph $K_n$ in $S_n$, it contributs $n$ to the maximal independence set of $\alpha\bar{(S_n)}$ and $U = \{u_1, u_2, u_3, . . . , u_n\}$ can't do anything regarding the same.  That is $\alpha\bar{(S_n)} = n$. By theorem \ref{th1}, $\beta\bar{(S_n)} = 2n-n=n$.
\\
Therefore, For a complete sun graph $S_n$, $n\geq 3$, $\beta(S_n)+\beta(\bar{S_n})= 2n$ and
$\beta(S_n).\beta(\bar{S_n})= n^2$.
\end{proof}

\begin{definition}{\rm
\cite{IS} The $n-$sunlet graph is the graph on $2n$ vertices obtained by attaching $n$ pendant edges to a cycle graph $C_n$ and is denoted by $L_n$.}
\end{definition}
\begin{theorem}{\rm
For a sunlet graph $L_n$ on $2n$ vertices, $n\geq 3$,
$\beta(L_n)+\beta(\bar{L_n})= n$ and
$\beta(L_n).\beta(\bar{L_n})= 2(n-1)$.}
\end{theorem}

\begin{proof}{\rm
Let $L_n$ be a sunlet graph on $2n$ vertices. By \cite{SS1}, the independence number of $L_n$, $\alpha(L_n)=n$.  Since $L_n$ consists of $2n$ vertices, by theorem \ref{th1}, $\beta(L_n)=2n-n=n$. By \cite{SS2}, $\alpha (\bar{L_n})=2$ and by theorem \ref{th1}, $\beta (\bar{L_n})=2n-2=2(n-1)$.
\\
Therefore, For a sunlet graph $L_n$ on $2n$ vertices, $n\geq 3$,
$\beta(L_n)+\beta(\bar{L_n})= n$ and
$\beta(L_n).\beta(\bar{L_n})= 2(n-1)$.}
\end{proof}
\begin{definition}{\rm
\cite{FH} The \textit{armed crown} is a graph $G$ obtained by adjoining a path $P_{m}$ to every vertex of a cycle $C_n$.}
\end{definition}

\begin{theorem}{\rm
For an armed crown graph $G$ with a path $P_m$ and a cycle $C_n$, 
$\beta(G)+\beta(\bar{G})=\frac{3mn-4}{2}$ and $\beta(G).\beta(\bar{G})=\frac{mn(mn-2)}{2}$.}
\end{theorem}

\begin{proof}
Note that the number of vertices of $P_m$ is $m$. 
By \cite{SS1}, $\alpha(G)=\frac{mn}{2}$, except for $m$ and $n$ are odd .  The number of vertices of an armed crown graph is $mn$.  By theorem \ref{th1}, $\beta(G)=mn-\frac{mn}{2}=\frac{mn}{2}$.  Now consider the complement of $G$.  Let $v$ be a pendent vertex of any one of the paths $P_m$ attached to the cycle $C_n$.  Also let $u$ be the vertex adjacent to $v$ in the path.  Then $u$ and $v$ are independent in $\bar{G}$ and all the other vertices in $P_m$ is adjacent either to $u$ or to $v$, since the armed crown contains no complete graph other than $K_2$.  There will be no independent set with cardinality greater than or equal to 3.  There fore the independence number of  $\bar{G}$ is 2.  By theorem \ref{th1}, $\beta(\bar{G})=mn-2$.
\\
Therefore For an armed crown graph $G$ with a path $P_m$ and a cycle $C_n$, 
$\beta(G)+\beta(\bar{G})=\frac{mn}{2}+(mn-2)=\frac{3mn-4}{2}$ and $\beta(G).\beta(\bar{G})=\frac{mn}{2}.(mn-2)=\frac{mn(mn-2)}{2}$.
\end{proof}
\section{Conclusion}{\rm
The theoretical results obtained in this research may provide a better insight into the problems involving matching number and independence number by improving the known lower and upper bounds on sums and products of independence numbers of a graph $G$ and an associated graph of $G$.  More properties and characteristics of operations on independence number and also other graph parameters are yet to be investigated.  The problems of establishing the inequalities on sums and products of independence numbers for various graphs and graph classes still remain unsettled. All these facts highlight a wide scope for further studies in this area.
\\\\This work is motivated by the inspiring talk given by Dr. J Paulraj Joseph, Department of Mathematics, Manonmaniam Sundaranar University,  TamilNadu, India titled \textbf{Bounds on sum of graph parameters - A survey}, at the National Conference on Emerging Trends in Graph Connections (NCETGC-2014), University of Kerala,  Kerala, India.}

\end{document}